\documentclass{sn-jnl}
\usepackage{graphicx}%
\usepackage{multirow}%
\usepackage{amsmath,amssymb,amsfonts,}%
\usepackage{mathtools}
\usepackage{amsthm}%
\usepackage{mathrsfs}%
\usepackage[title]{appendix}%
\usepackage{xcolor}%
\usepackage{textcomp}%
\usepackage{manyfoot}%
\usepackage{booktabs}%
\usepackage{algorithm}%
\usepackage{algorithmicx}%
\usepackage{algpseudocode}%
\usepackage{listings}%
\usepackage{hyperref}
\usepackage{bm}
\usepackage{dsfont}
\usepackage[labelfont=bf,labelsep=space]{caption}
\usepackage{enumitem}
\usepackage{titlesec}
\usepackage{caption}
\usepackage{float}
\usepackage{subcaption}

\newtheorem*{theorem*}{Theorem}
\theoremstyle{thmstyleone}%
\newtheorem{theorem}{Theorem}

\newtheorem{proposition}[theorem]{Proposition}%
\newtheorem{corollary}[theorem]{Corollary}%
\newtheorem{lemma}[theorem]{Lemma}%

\theoremstyle{thmstyletwo}%
\newtheorem{example}{Example}%
\newtheorem{remark}{Remark}%

\theoremstyle{thmstylethree}%
\usepackage[backend=bibtex, giveninits=true]{biblatex} 
\AtBeginBibliography{%
}
\begin{document}
\nocite{*}
\newcommand{\R}{\mathbb{R}}
\newcommand{\N}{\mathbb{N}}
\newcommand{\So}{\mathcal{S}}
\newcommand{\Ne}{\mathcal{N}}
\newcommand{\M}{M^{+}(\R^n)}
\newcommand{\Linf}{L^{\infty}(\Om)}
\def\wpp{W^{1,p}(\Omega)}
\newcommand{\hs}{H^{s}(\R^{n})}
\def\O{\Omega}
\def\bO{\partial\O}
\newcommand{\smoo}{\mathcal{C}^{\infty}_{0}(\Om)}
\newcommand{\ti}[1]{\textit{#1}}
\newcommand{\e}{\epsilon}
\newcommand{\g}{\gamma}
\newcommand{\G}{\Gamma}
\newcommand{\al}{\alpha}
\newcommand{\be}{\beta}
\newcommand{\te}{\theta}
\newcommand{\se}{\subseteq}
\newcommand{\si}{\sigma}
\newcommand{\un}{u_{n}}
\newcommand{\vn}{v_{n}}
\newcommand{\vf}{\varphi}
\newcommand{\wc}{\rightharpoonup}
\newcommand{\phik}{\Phi^{k}(\Om)}
\newcommand{\la}{\lambda}
\newcommand{\dx}{\,dx}
\newcommand{\dha}{\,d\mathcal{H}^{N-1}}
\newcommand{\Div}{\mathrm{div}}
\newcommand{\supp}{\mathrm{supp}}
\newcommand{\iom}{\int_{\Omega}}
\newcommand{\plap}{\Delta_{p}}
\def\ih#1{\int_{\bO}#1\dha}
\def\io#1{\int_{\O}#1\dx}
\def\L#1{L^{#1}(\R^{n})}
\def\norma#1#2{\|#1\|_{\lower 4pt \hbox{$\scriptstyle #2$}}}
\newcommand{\cs}{\subset\joinrel\subset}
\newcommand{\eq}[2]{\begin{equation} \label{eq:#1} #2 \end{equation}}
\newcommand{\thm}[2]{\begin{theorem} \label{thm:#1} #2 \end{theorem}}
\newcommand{\prop}[2]{\begin{proposition} \label{prop:#1} #2 \end{proposition}}
\newcommand{\cor}[2]{\begin{corollary} \label{cor:#1} #2 \end{corollary}}
\newcommand{\lem}[2]{\begin{lemma} \label{lem:#1} #2 \end{lemma}}
\newcommand{\ex}[2]{\begin{example} \label{ex:#1} #2 \end{example}}
\newcommand{\rem}[2]{\begin{remark} \label{rem:#1} #2 \end{remark}}
\newcommand{\pf}[1]{\begin{proof}  #1 \end{proof}}
\newcommand{\case}[1]{\begin{cases}  #1 \end{cases}}
\newcommand{\splitme}[1]{\begin{split}  #1 \end{split}}
\newcommand{\eps}{\varepsilon}
\newcommand{\de}{\partial}
\def\HH{\mathcal{H}}
\def\ds{\displaystyle}

\title{A nonlinear problem related to optimal insulation}
\author{\fnm{Genival} \sur{da Silva}\footnote{email: gdasilva@tamusa.edu, website: \url{www.gdasilvajr.com}}}
\affil{\orgdiv{Department of Mathematics}, \orgname{Texas A\&M University - San Antonio}}


\abstract{
We study a generalization of the classical optimal insulation problem by replacing the standard Laplacian with the \ti{p-Laplacian}, leading to a \ti{p-Poisson equation} with Robin boundary conditions. We also investigate the corresponding eigenvalue problem. This work extends and complements some results from \cite{bu1,pietra1,pietra2} to the nonlinear setting.
}

\keywords{p-Laplacian, Robin boundary conditions, optimal insulation}


\pacs[MSC Classification]{35J25,35J92,49R05}

\maketitle
\section{Introduction}
In this work, we investigate a generalization of the classical optimal insulation problem. Broadly speaking, the classical problem seeks the optimal distribution of an insulating material around a fixed thermally conducting body.

Mathematically, the fixed conducting body is represented by an open, bounded set $\Omega \subset \mathbb{R}^{n}$ with Lipschitz boundary, and the insulating material is modeled by a set $\Sigma_{\eps} \subset \mathbb{R}^{n}$ defined by
\[
\Sigma_{\eps} = \{\sigma + t\nu(\sigma) \,:\, \sigma \in \partial\Omega,\, 0 \le t < \eps^{\frac1{p-1}} h(\sigma)\},
\]
where $\nu(\sigma)$ denotes the unit outward normal (which is well-defined since $\Omega$ is assumed to have a Lipschitz boundary), and $h:\bO\to \R$ is a bounded positive Lipschitz  function.

Set $\Omega_{\varepsilon} = \Omega \cup \Sigma_{\varepsilon}$. If $u(x)$ denotes the temperature at a point $x \in \Omega_{\varepsilon}$, and $f$ is a given source function, then $u\in H^{1}_{0}(\O_{\eps})$ minimizes the functional
\eq{energy}{
F_{\varepsilon}(u) = \frac{1}{2} \int_{\Omega} \left| D u \right|^{2} \, dx + \frac{\varepsilon}{2} \int_{\Sigma_{\varepsilon}} \left| D u \right|^{2} \, dx - \int_{\Omega} f u \, dx.
}
Equivalently, $u$ is a solution to the corresponding Euler–Lagrange equation:
\begin{equation}\label{eq:euler}
\begin{cases}
-\Delta u = f & \text{in } \Omega, \\
-\Delta u = 0 & \text{in } \Sigma_{\varepsilon}, \\
u = 0 & \text{on } \partial \Omega_{\varepsilon}, \\
\frac{\partial u^{-}}{\partial \nu} = \varepsilon \frac{\partial u^{+}}{\partial \nu} & \text{on } \partial \Omega.
\end{cases}
\end{equation}
where $u^{-}$ and $u^{+}$ denote the traces of $u$ in $\Omega$ and $\Sigma_{\varepsilon}$ respectively. 

The optimal insulation problem consists in studying the behavior of the solution $u$ as $\varepsilon \to 0$. In this context, the notion of $\Gamma$-convergence in the $L^{2}$ topology is particularly well-suited for analyzing the limiting behavior of the functionals $F_{\varepsilon}$ as $\varepsilon \to 0$.

In \cite{ca2,but1}, the following result is proved:
\begin{theorem*}The functionals $F_{\eps}$ defined by \eqref{eq:energy} $\Gamma$-converge in the $L^{2}$ topology to the functional
\[
F(u)=\frac{1}{2} \int_{\Omega} \left| D u \right|^{2} \, dx + \frac{1}{2} \int_{\bO} \frac{u^{2}}{h} \, d\HH^{N-1} - \int_{\Omega} f u \, dx.
\]
Consequently, the unique minimizer of $F$ satisfies the Euler--Lagrange equation:
\eq{}{
\begin{cases}
-\Delta u = f & \text{in } \Omega, \\
h\frac{\partial u}{\partial \nu} +u =0  & \text{on } \partial \Omega.
\end{cases}
}
\end{theorem*}
\begin{figure}
\centering
    \includegraphics[width=.6\linewidth]{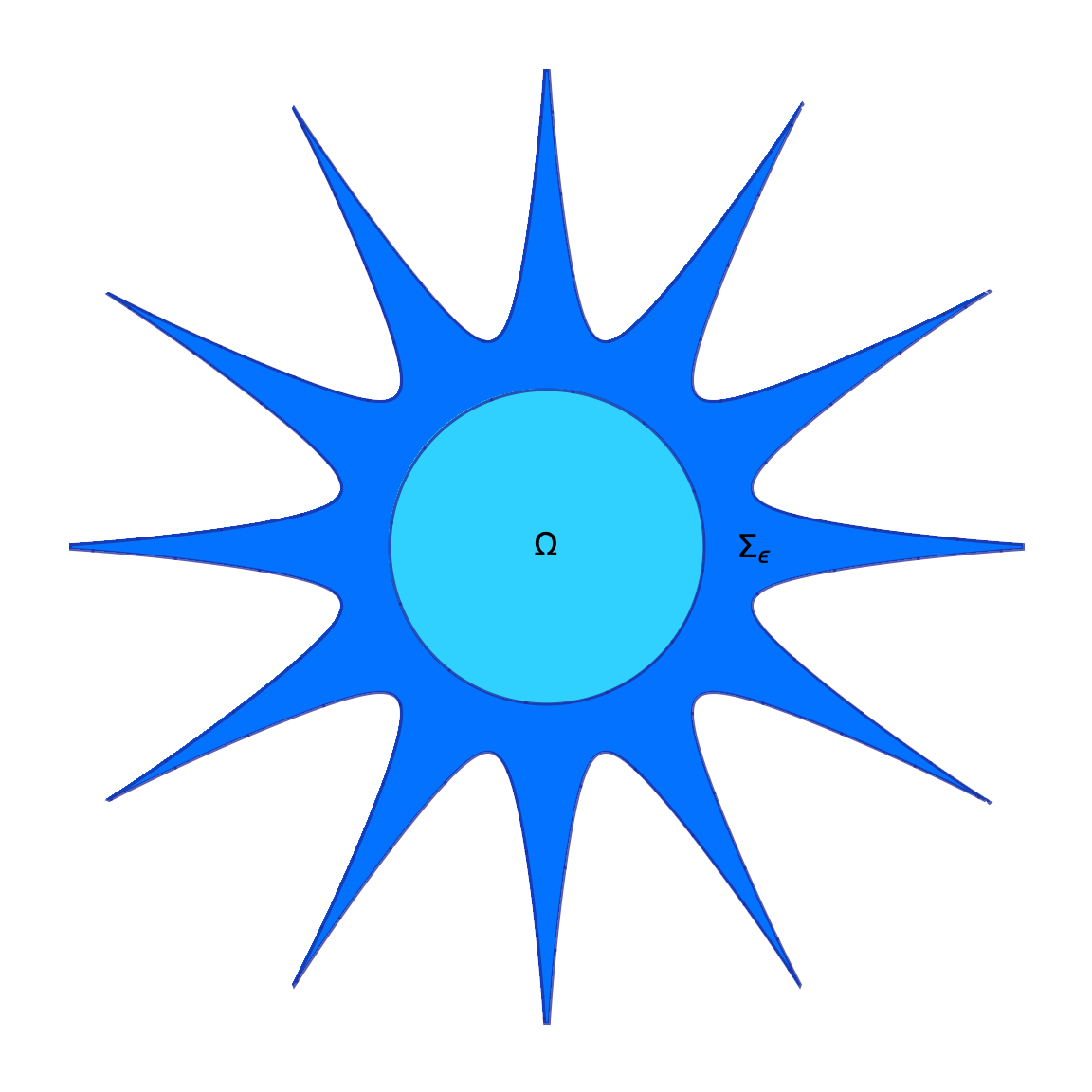}
\caption{Representation of  disk $\O\se\R^{2}$ with insulator described by the set $\Sigma_{\e}$. The function $h$ describing the contour is not optimal for $F_{h}$ if $p=2$.}
\end{figure}
In this work, we plan to discuss a generalized version of the functional $F$, namely:
\[
F_{h}(u)=\frac{1}{p}\iom \left| D u\right|^{p}dx+\frac{1}{p}\ih{\frac{|u|^{p}}{h^{p-1}}}-\iom fu
\]
where $p>1$. For a fixed insulation profile we first take $h:\partial\Omega\to(0,\infty)$ to be bounded and Lipschitz, so that the boundary term and the associated Robin condition have their usual pointwise meaning. In the optimization problem below, however, the admissible class is enlarged to measurable nonnegative functions. In that case the boundary term is understood in the extended-valued sense described below. The associated Euler-Lagrange equation is given by
\eq{el}{
\begin{cases}
-\Delta_{p} u = f & \text{in } \Omega, \\
h^{p-1}|Du|^{p-2}\frac{\partial u}{\partial \nu} +|u|^{p-2}u =0  & \text{on } \partial \Omega.
\end{cases}
}

With the thickness scaling $r_\varepsilon=\varepsilon^{1/(p-1)}$ used in the definition of $\Sigma_\varepsilon$, the functional $F_h$ is also related to the optimal insulation problem through $\Gamma$-convergence. More precisely, the relevant ratio satisfies $\varepsilon/r_\varepsilon^{p-1}=1$, and the limiting boundary coefficient is therefore finite and nonzero. With this scaling, $F_h$ arises as the $\Gamma$-limit of the family of functionals (see \cite{but1})
\[
F_{h}^{\eps}(u)=\frac{1}{p}\iom \left| D u\right|^{p}dx+\frac{\e}{p}\int_{\Sigma_{\eps}} \left| D u\right|^{p}dx-\iom fu\dx,
\]
which coincides with \eqref{eq:energy} when $p=2$.

When \( p \neq 2 \), the \( p \)-Laplacian still appears in various physical models, such as image denoising \cite{img}, sandpile modeling \cite{sandpile}, and modeling of non-Newtonian fluids \cite{flu}. However, it is no longer directly related to heat transfer. Nonetheless, the mathematical analysis of \( F_h \) remains of interest even in the case \( p \neq 2 \).

Our first result addresses the following question:

\vspace{2mm}

\textbf{Question:} What is the optimal pair  \( (u,h) \) that minimizes the value of \( F_h(u) \), where \( u \in W^{1,p}(\Omega) \) and \( h \) has fixed content, i.e., \( \int_{\partial\Omega} h\,d\HH^{N-1} = m \)?

\vspace{2mm}

In other words, we aim to analyze the following double minimization problem:
\eq{main}{
\min\limits_{h\in\HH_{m}}\min\limits_{ u\in \wpp}\left\{F_{h}(u)\right\},
}
where 
\[
\HH_m=\left\{h:\partial\O\to\R\hbox{ measurable, }h\ge0,\ \int_{\partial\O}h\,d\HH^{N-1}=m\right\}.
\]
For $h\in\HH_m$ we interpret
\[
\frac{|u|^p}{h^{p-1}}=
\begin{cases}
|u|^p h^{1-p},& h>0,\\
0,& h=0\ \text{and }u=0,\\
+\infty,& h=0\ \text{and }u\ne0,
\end{cases}
\qquad\text{on }\partial\Omega.
\]
Thus $F_h$ is an extended-valued functional on $W^{1,p}(\Omega)$. This convention is the natural lower-semicontinuous interpretation of the boundary energy when the insulation is allowed to vanish.
In the case \( p = 2 \), this question was addressed in \cite{bu1}. The present paper can be seen as a natural generalization of that work to the case \( p \neq 2 \). In this setting, the problem becomes nonlinear, and certain adjustments are required due to the presence of the nonlinearity \( |Du|^{p-2} \) in the definition of the \( p \)-Laplacian.

The principal new point of the paper is the isoperimetric estimate in Theorem~\ref{thm:54}. Although its level-set strategy is inspired by the linear argument, the nonlinear flux $|Du|^{p-2}Du$ changes the power structure in the coarea--Hölder step and requires a genuinely $p$-dependent estimate. 

In the discussion above, the set \( \Omega \) is assumed to be fixed. It is natural to ask how the minimum of \( F_h \) behaves when both \( h \) and \( \Omega \) are allowed to vary in an appropriate sense. Our second  result below addresses this scenario.  More precisely, we will show that among all Lipschitz domains with fixed volume, the ball minimizes the value of \( F_h \). This result constitutes a type of isoperimetric inequality and extends the work in \cite[Corollary~5.1]{pietra1}, where the linear case is considered.
\begin{center}
\begin{figure}
    \centering
    \begin{subfigure}[t]{0.5\textwidth}
        \centering
        \includegraphics[height=1in]{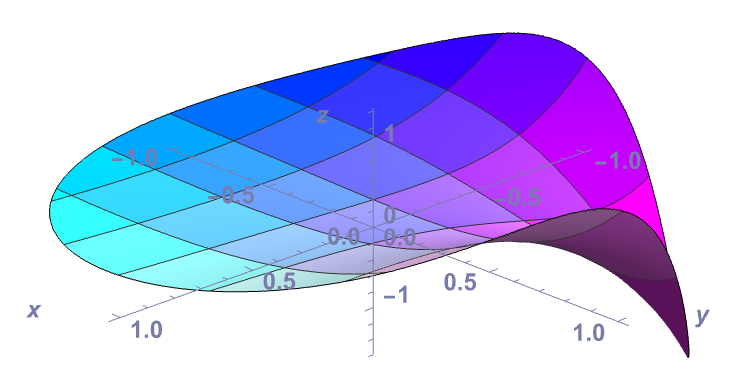}
    \end{subfigure}%
    ~ 
    \begin{subfigure}[t]{0.5\textwidth}
        \centering
        \includegraphics[height=1in]{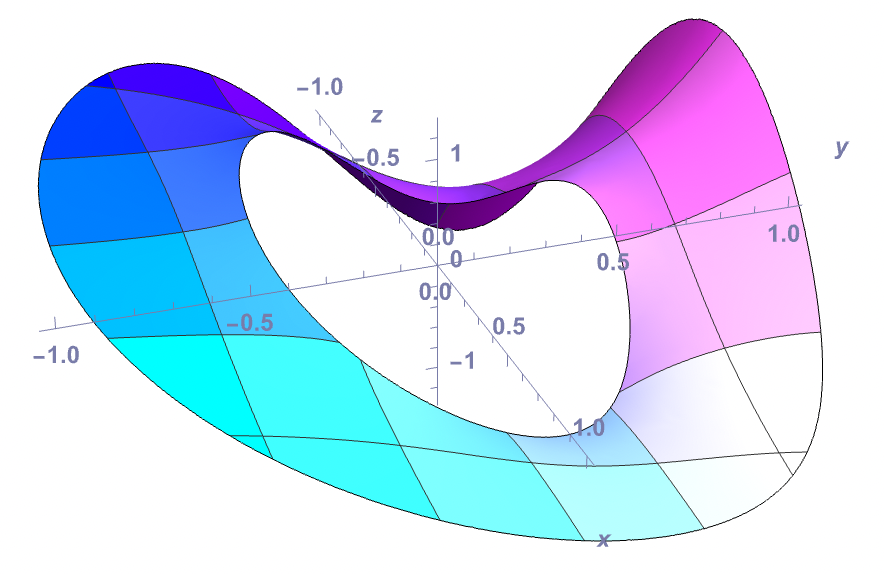}
    \end{subfigure}
   \caption{Graph of the solution \( u(x, y) \) to problem \eqref{eq:10} with \( p = 2 \) and \( h(x) = e^{x - y} \). 
    Left: \( \Omega = B_1 \subset \mathbb{R}^2 \); 
    Right: \( \Omega = B_1 \setminus \overline{B_{1/2}} \).}
\end{figure} 
In the last section of this paper, we study the eigenvalue problem:
\end{center}
\eq{10}{
\begin{cases}
-\Delta_{p} u = \la |u|^{p-2}u & \text{in } \Omega, \\
h^{p-1}|Du|^{p-2}\frac{\partial u}{\partial \nu} +|u|^{p-2}u =0  & \text{on } \partial \Omega.
\end{cases}
}

Using the direct method in the calculus of variations, it is straightforward to show that \eqref{eq:10} admits a solution \( u \in W^{1,p}(\Omega) \). Specifically, define the functional
\[
J_h(u) = \frac{\displaystyle\int_{\Omega} \left| \nabla u \right|^p\, dx + \int_{\partial \Omega} \frac{|u|^p}{h^{p-1}}\, d\mathcal{H}^{n-1}}{\displaystyle\int_{\Omega} |u|^p\, dx},
\]
then any minimizer of
\[
\min\left\{ J_h(u) : u \in W^{1,p}(\Omega),\, u \neq 0 \right\}
\]
solves \eqref{eq:10}.

Our second result concerns the solution of the double minimization problem:
\begin{equation}\label{eq:main2}
\min\limits_{h \in \mathcal{H}_m} \ \min\limits_{u \in W^{1,p}(\Omega)} J_h(u),
\end{equation}
and our final result is the analysis of the same double minimization problem when we also allow the Lipschitz domain \( \Omega \subset \mathbb{R}^n \) to vary. That is, we consider the following minimization problem:
\begin{equation}\label{eq:main3}
\min\limits_{|\Omega| \le 1} \ \min\limits_{h \in \mathcal{H}_m} \ \min\limits_{u \in W^{1,p}(\Omega)} J_h(u).
\end{equation}

\subsection*{Notation \& Assumptions}
\begin{itemize}
\item[-] $\O\se \R^{N}$ is a bounded set with Lipschitz boundary.
\item[-] The space $\wpp$ denotes the usual Sobolev space.
\item[-] For $q>1$, $q'$ denotes the Hölder conjugate, i.e. $\frac{1}{q}+\frac{1}{q'}=1$. When $1<q<N$, $q^{*}=\frac{qN}{N-q}$ denotes the Sobolev conjugate. 
\item[-] The letter $C$ will always denote a positive constant which may vary from place to place.
\item[-] The Lebesgue measure of a set $A\subseteq \R^{n}$ is denoted by $|A|$.
\item[-] The $N-1$ Hausdorff measure of a set $A\subseteq \R^{n}$ is denoted by $\HH^{N-1}(A)$.
\item[-] The symbol $\rightharpoonup$ denotes weak convergence. 
\end{itemize}
\section{Solution to the double minimization problem \eqref{eq:main}}
The following lemma will be used in the proof below.
\lem{poin}{(Poincaré inequality) Let $u\in \wpp$. Then there exists a constant $C>0$ such that
\eq{0}{
\io{|u|^{p}}\le C\left[ \iom \left| D u\right|^{p}dx+\left(\ih{|u|}\right)^{p}\right].
}
}
\pf{
Suppose \eqref{eq:0} is false. Then there is a sequence $\un\in\wpp$ such that
\[
\io{|\un|^{p}}> n\left[ \iom \left| D \un\right|^{p}dx+\left(\ih{|\un|}\right)^{p}\right].
\]
Without loss of generality we may assume 
\[
\io{|\un|^{p}}=1.
\]
Hence, up to a subsequence, $\un\wc u$ in $\wpp$, with $\io{|u|^{p}}=1$.

In particular,
\[
\iom \left| D \un\right|^{p}dx+\left(\ih{|\un|}\right)^{p} \to 0
\]
when $n\to+\infty$.  But since 
\[
\iom \left| D \un\right|^{p}dx\to 0\text{ and } \left(\ih{|\un|}\right)^{p} \to 0
\]
one must have $D\un\to Du=0$ strongly in $L^{p}(\O)$, and  $\un\to u=0$ strongly in $L^{p}(\bO)$. Hence $u\equiv 0$, a contradiction.
}

\thm{1}{Assume that $\O$ is connected, $f\in L^{p'}(\Omega)$, $f\ge0$, and $f\not\equiv0$. Then the minimization problem \eqref{eq:main} admits a unique solution. In particular, if $\O=B_{R}$ and $f=1$, then the optimal solution $h(x)$ is constant, given by
\[
h(x)=\frac{m}{N \omega_{N} R^{N-1}}.
\]
and does not depend on $p$.
}
\pf{
Note that given $u \in L^{p}(\bO)$ with $u \not\equiv 0$, the minimization problem
\[
\min\left\{\ih{\frac{|u|^p}{h^{p-1}}}\ :\ h\in \HH_{m}\right\}
\]
admits a unique solution. Indeed, define
\[
\hat{h}=m\frac{|u|}{\left(\ih{|u|}\right)}.
\]
Then, by Hölder's inequality,
\[
\left(\ih{|u|}\right)^{p}\le \left(\ih{\frac{|u|^{p}}{h^{p-1}}}\right)\left(\ih h\right)^{p-1},
\]
which implies
\[
\ih{\frac{|u|^{p}}{h^{p-1}}}\ge \frac{1}{m^{p-1}}\left(\ih{|u|}\right)^{p}=\ih{\frac{|u|^{p}}{\hat{h}^{p-1}}}.
\]
Thus, $\hat{h}$ is a minimizer. Uniqueness follows from the strict convexity of the functional $G(h)=\ih{\frac{|u|^p}{h^{p-1}}}$.

It follows that the minimization problem \eqref{eq:main} is equivalent to
\eq{3}{\min\left\{\frac{1}{p}\iom \left| D u\right|^{p}dx+\frac{1}{m^{p-1}p}\left(\ih{|u|}\right)^{p}-\iom fu \ :\ u\in \wpp \right\}.}
Notice that the functional
\[
R(u)=\frac{1}{p}\iom \left| D u\right|^{p}dx+\frac{1}{m^{p-1}p}\left(\ih{|u|}\right)^{p}-\iom fu 
\]
is coercive by Lemma \ref{lem:poin}, albeit not differentiable. Additionally, the third term is trivially convex. 

Existence follows from Lemma~\ref{lem:poin}, the direct method, and the weak lower semicontinuity of the two convex terms in \eqref{eq:3}. We next prove uniqueness under the stated sign assumption on $f$.

First observe that every minimizer may be chosen nonnegative. Indeed, for every $u\in W^{1,p}(\Omega)$,
\[
|D|u||=|Du|\quad\text{a.e. in }\Omega,\qquad
\int_{\partial\Omega}|\,|u|\,|\,d\HH^{N-1}=\int_{\partial\Omega}|u|\,d\HH^{N-1},
\]
while, since $f\ge0$,
\[
-\int_\Omega f|u|\,dx\le -\int_\Omega fu\,dx.
\]
Thus $R(|u|)\le R(u)$. If $u$ is a minimizer, equality must hold, hence
\[
\int_\Omega f(|u|-u)\,dx=0.
\]
Moreover, $|u|$ is itself a minimizer and hence satisfies $-\Delta_p |u|=f$ in the weak sense in $\Omega$. The strong maximum principle gives $|u|>0$ in $\Omega$. Thus $u$ has a constant sign in the connected set $\Omega$; the identity $\int_\Omega f(|u|-u)\,dx=0$ and the assumption $f\ge0$, $f\not\equiv0$ exclude the negative sign. Hence every minimizer is nonnegative (indeed positive in $\Omega$).

Let $u$ and $v$ be two minimizers. By convexity, $(u+v)/2$ is also a minimizer. Equality in the strictly convex gradient term gives
\[
Du=Dv\quad\text{a.e. in }\Omega,
\]
and, because $\Omega$ is connected, $u-v=C$ for some constant $C$. Since $u,v\ge0$, the boundary term reduces to
\[
\left(\int_{\partial\Omega}u\,d\HH^{N-1}\right)^p
\quad\text{and}\quad
\left(\int_{\partial\Omega}v\,d\HH^{N-1}\right)^p.
\]
Equality in the convexity inequality for the boundary term therefore implies
\[
\int_{\partial\Omega}u\,d\HH^{N-1}
=
\int_{\partial\Omega}v\,d\HH^{N-1}.
\]
Using $u-v=C$ gives $C\,\HH^{N-1}(\partial\Omega)=0$, hence $C=0$ and $u=v$. The corresponding optimal profile
\[
h=m\frac{u}{\int_{\partial\Omega}u\,d\HH^{N-1}}
\]
is then unique as well.

Let us now assume that $\O=B_{R}$ and $f=1$. A straightforward computation shows that the radial function
\eq{bsol}{
u(x)=\frac{R^{p'}-|x|^{p'}}{p' N^{\frac1{p-1}}}
+\frac{m}{\omega_{N}N^{p'}R^{N-p'}}
}
is the unique solution to the Euler-Lagrange equation associated to \eqref{eq:3},  which reads
\[
\begin{cases}
-\Delta_{p} u=1\quad\mbox{in }B_{R},\\
0\in m^{p-1}|Du|^{p-2}\ds\frac{\partial u}{\partial\nu}+\phi(u)\left(\int_{\partial B_{R}}u\, d\mathcal{H}^{n-1}\right)^{p-1}\quad\mbox{on }\partial B_{R},\\
\end{cases}
\]
where $\phi(u)$ is the subdifferential  of the real valued function $|t|,\, t=\phi(u)$. 

Note that the appearance of the inclusion $0\in \ldots$ in the Euler-Lagrange equation is due to the non differentiability of the functional \[u\mapsto \ih{|u|}.\]

We conclude that when $\O=B_{R}$ we have
\[
h(x)=m\frac{|u(x)|}{\left(\ih{|u|}\right)}=\frac{m}{N \omega_{N} R^{N-1}}.
\]
}
\begin{figure}
\centering
    \includegraphics[width=.4\linewidth]{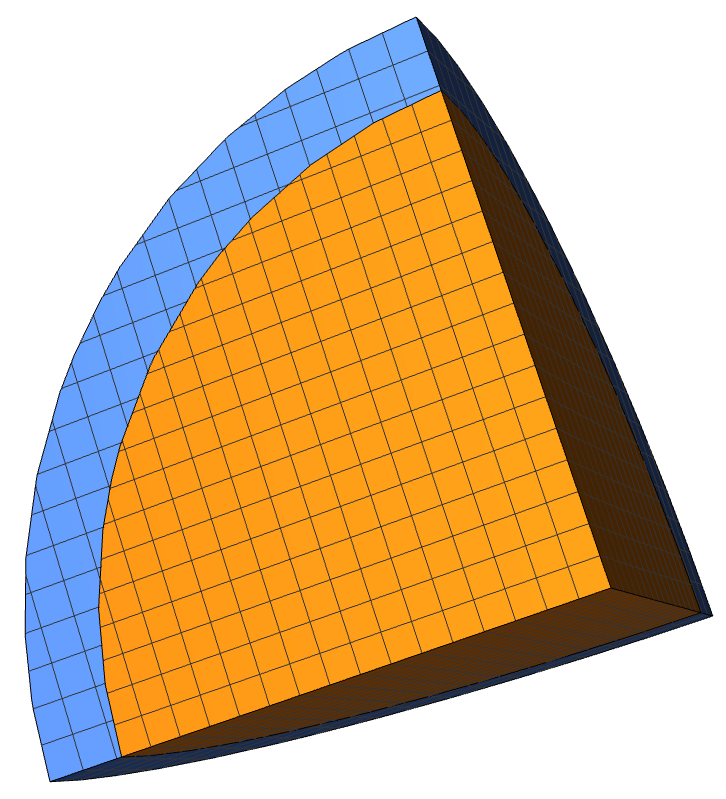}
\caption{Representation of $B_{1}$ in $\R^{3}$ with insulator material of constant thickness $h(x)=\frac{1}{4\pi}$.}
\end{figure}

For completeness, when $h>0$ a.e. on $\partial\Omega$ and the boundary energy is finite, by a weak solution of \eqref{eq:el} we mean a function $u\in W^{1,p}(\Omega)$ such that
\[
\int_\Omega |Du|^{p-2}Du\cdot D\varphi\,dx
+\int_{\partial\Omega}\frac{|u|^{p-2}u}{h^{p-1}}\varphi\,d\HH^{N-1}
=\int_\Omega f\varphi\,dx
\]
for every admissible $\varphi\in W^{1,p}(\Omega)$ for which the boundary integral is finite. If $h$ vanishes, the variational formulation is understood through the extended-valued functional above; in particular, finite energy forces the trace of $u$ to vanish a.e. on $\{h=0\}$, and variations are taken within this effective domain.

It is possible to obtain a closed-form expression for \[E(h)=\min_{ u\in \wpp} F_{h}(u). \] Specifically, if \( u \) is a solution of \eqref{eq:main}, then by taking \( u \) as a test function in the variational identity above, one obtains
\[
E(h) = \frac{1 - p}{p} \int_\Omega f u.
\]
In particular, when \( f \equiv 1 \), minimizing \( E(h) \) is equivalent to maximizing the average value of \( u \) over \( \Omega \). In the linear case, $p=2$, the solution $u$ is a model for temperature, and in this case, the solution maximizes the average temperature in $\O$ (see \cite{bu1} for more).

A natural question is: what is the optimal domain for the minimization problem \eqref{eq:main}? That is, if we allow the domain \( \Omega \) to vary among all Lipschitz domains with fixed volume, which one minimizes \eqref{eq:main}? This problem can be seen as a type of isoperimetric inequality and remains open for \( p \neq 2 \) (it has been resolved in the case \( p = 2 \); see \cite{pietra1}). The following theorem answers this question for every $p>1$.

\thm{54}{
Suppose $p>1$, $f\equiv 1$, $m>0$ fixed, and let $(u,h)$ be the solution pair obtained in Theorem \ref{thm:1}. Then
\[
\iom{u}\, \dx\le \frac{1}{N^{p'}\omega_{N}^{\frac{p'}{N}}}\left(\frac{N(p-1)}{p+N(p-1)} |\O|^{\frac{p+N(p-1)}{N(p-1)}}+ |\O|^{\frac{p}{N(p-1)}}m\right),
\]
and equality holds if $\O$ is a ball. 
}
\pf{
First notice that $u\ge0$ a.e. in $\Omega$. This follows directly from the variational problem and does not require a boundary point lemma: replacing $u$ by $|u|$ leaves both the gradient term and the boundary term unchanged, while
\[
-\int_\Omega |u|\,dx\le -\int_\Omega u\,dx
\]
because $f\equiv1$. Hence an optimal state can be taken nonnegative; by the uniqueness established in Theorem~\ref{thm:1}, the minimizer itself is nonnegative. The strong maximum principle for
$-\Delta_pu=1$ then gives
\[
u>0\quad\text{in }\Omega.
\]

For $t>0$, define
\[
\splitme{
U_{t}&=\{x\in\O;u(x)>t\},\\
\partial U_{t}^{\textrm{int}}&=\partial U_{t}\cap \O,\text{ and }\partial U_{t}^{\textrm{ext}}=\partial U_{t}\cap \bO,\\
\mu(t)&=|U_{t}|,\\
P(t)&=\text{Per}(U_{t}).
}
\]
Now, given $t,k>0$, consider the test function $\vf$ in the variational formulation of \eqref{eq:el}, defined as
\[
\vf=\case{0,& \text{if } 0<u<t, \\
u-t,& \text{if } t<u<t+k, \\
k,& \text{if } u>t+k, 
}
\]
We obtain
\[
\splitme{
&\int_{U_{t}\setminus U_{t+k}} |Du|^{p}\dx +k \int_{\partial U_{t+k}^{\textrm{ext}}} \frac{|u|^{p-2}u}{h^{p-1}}\dha\\&+\int_{\partial U_{t}^{\textrm{ext}}\setminus\partial U_{t+k}^{\textrm{ext}}} \frac{|u|^{p-2}u}{h^{p-1}}(u-t)\dha= \int_{U_{t}\setminus U_{t+k}} (u-t)\dx + k\int_{U_{t+k}} \dx 
}
\]
Dividing both sides by $k$ and letting $k\to 0$ we have
\[
\mu(t)=\int_{\partial U_{t}^{\textrm{int}}} |Du|^{p-1}\dha+\int_{\partial U_{t}^{\textrm{ext}}} \frac{|u|^{p-2}u}{h^{p-1}}\dha.
\]
If we set 
\[
g(x)=\case{ |Du|^{p-1},& \text{if } x\in \partial U_{t}^{\textrm{int}},\\ \frac{|u|^{p-2}u}{h^{p-1}},& \text{if } x\in \partial U_{t}^{\textrm{ext}}},
\]
then the above expression becomes
\[
\mu(t)=\int_{\partial U_{t}} g\dha.
\]
Applying Hölder’s inequality, we obtain
\eq{20}{
P(t)^p
\le
\left(\int_{\partial U_t} g\,d\HH^{N-1}\right)
\left(\int_{\partial U_t}g^{-\frac1{p-1}}\,d\HH^{N-1}\right)^{p-1}.
}
Since $\int_{\partial U_t}g\,d\HH^{N-1}=\mu(t)$, taking the power $1/(p-1)$ in \eqref{eq:20} gives
\[
P(t)^{p'}
\le
\mu(t)^{\frac1{p-1}}
\int_{\partial U_t}g^{-\frac1{p-1}}\,d\HH^{N-1}.
\]
By the coarea formula on $\partial U_t^{\mathrm{int}}$ and by the definition of $g$ on $\partial U_t^{\mathrm{ext}}$,
\[
\int_{\partial U_t}g^{-\frac1{p-1}}\,d\HH^{N-1}
=
-\mu'(t)+\int_{\partial U_t^{\mathrm{ext}}}\frac{h}{u}\,d\HH^{N-1}.
\]
Therefore
\[
P(t)^{p'}
\le
\mu(t)^{\frac1{p-1}}
\left(
-\mu'(t)+\int_{\partial U_t^{\mathrm{ext}}}\frac{h}{u}\,d\HH^{N-1}
\right).
\]
This application of Hölder's inequality uses the conjugate exponents $p$ and $p'=p/(p-1)$ and is valid for every $p>1$.
Recall the isoperimetric inequality:
\eq{}{
\left(\frac{\mu(t)}{\omega_{N}}\right)^{N-1}\le \left(\frac{P(t)}{N\omega_{N}}\right)^{N}.
}
Combining this with the preceding estimate, we obtain
\eq{pp}{
N^{p'}\omega_{N}^{\frac{p'}{N}}
\mu(t)^{1-\frac{p'}{N}}
\le
-\mu'(t)+ \int_{\partial U_{t}^{\textrm{ext}}} \frac{h}{u}\dha,
}
or equivalently,
\eq{}{
\mu(t)\le\frac{1}{N^{p'}\omega_{N}^{\frac{p'}{N}}}\left( -\mu'(t)\mu(t)^{\frac{p}{N(p-1)}}+ \mu(t)^{\frac{p}{N(p-1)}}\int_{\partial U_{t}^{\textrm{ext}}} \frac{h}{|u|}\dha\right).
}
Finally, integrating from $0$ to $+\infty$, we have
\eq{}{
 \iom u\dx=\int_{0}^{+\infty}\mu(t)\,dt  \le \frac{1}{N^{p'}\omega_{N}^{\frac{p'}{N}}}\left(\frac{N(p-1)}{p+N(p-1)} |\O|^{\frac{p+N(p-1)}{N(p-1)}}+ |\O|^{\frac{p}{N(p-1)}}m\right).
}
The explicit solution provided in \eqref{eq:bsol} achieves equality in the inequality above.
}

\rem{}{Although the above proof is not valid for \( p =1 \) due to division by zero (e.g in equation \eqref{eq:pp}), we expect that a modified version of the argument can be developed to address that scenario as well.
}
\section{The eigenvalue problem}
The results in this section are related and intersect with the ones presented in \cite{pietra2}. 
\begin{theorem}
The minimization problem \eqref{eq:main2} admits a solution.
\end{theorem}

\begin{proof}
The proof of Theorem~\ref{thm:1} can be readily adapted to this setting. Specifically, the problem is equivalent to minimizing the functional
\[
J(u) = \frac{\displaystyle \int_\Omega \left| D u \right|^p \, dx + \frac{1}{m^{p-1}} \left(\ih{|u|} \right)^p}{\displaystyle \int_\Omega |u|^p \, dx},
\]
which admits a minimizer by standard arguments from the Calculus of Variations. On the other hand, since the functional \( J(u) \) is not strictly convex, uniqueness of the minimizer is not guaranteed.

\end{proof}
As discussed above, an interesting open problem is to analyze the minimization problem \eqref{eq:main2} when the Lipschitz domain \( \Omega \) is allowed to vary. 

Remarkably, the following non-existence result holds for certain values of \( p \) in this setting.

\begin{theorem}
If \( p' < N \), then the minimization problem \eqref{eq:main3} does not admit a solution.
\end{theorem}

\begin{proof}
As before, the problem is equivalent to the minimization problem
\[
X = \inf \left\{ \frac{\displaystyle \int_\Omega |D u|^p\, dx + \frac{1}{m^{p-1}} \left( \int_{\partial\Omega} |u| \right)^p}{\displaystyle \int_\Omega |u|^p\, dx} : u \in W^{1,p}(\Omega) \setminus \{0\},\ |\Omega| \le 1 \right\}.
\]
Taking \( u \equiv 1 \) and considering the sequence \( \Omega_k = B_{1/k} \), we obtain
\[
X \le \frac{\left( \mathcal{H}^{N-1}(\partial B_{1/k}) \right)^p}{m^{p-1} |B_{1/k}|} 
= \frac{N^p \omega_N^{p - 1}}{m^{p - 1} k^{(N- 1)p - N}}.
\]
Letting \( k \to +\infty \), it follows that \( X = 0 \). Hence, the infimum is not attained by any open set \( \Omega \subset \mathbb{R}^N \), and the minimization problem \eqref{eq:main3} does not admit a solution.
\end{proof}
\rem{}{When \( p = 2 \),  we recover the result known in the linear case. Similarly, if \( p \ge 3 \), then \eqref{eq:main3} admits no solution for \( N \ge 2 \). On the other hand, when
\[
2 \le N \le  p',
\]
we believe that the arguments used in the proof of Theorem~\ref{thm:54} can be adapted to establish a Faber–Krahn-type isoperimetric inequality in this regime (see \cite{dai} for a related problem). In particular, we conjecture that the ball remains a minimizer in this setting as well. Complete proofs related to this Faber–Krahn-type isoperimetric inequality will be presented elsewhere.
}
\printbibliography
\end{document}